\documentclass[10pt]{amsart}
\usepackage[pagebackref,colorlinks=true]{hyperref}  %za cvetni linkove (obache bavi)
\usepackage{amsfonts}
\usepackage{amsmath,amsthm,amssymb,amscd,enumerate,eucal,url}
\usepackage{eucal,url,amssymb,verbatim,%booktabs,%
enumerate,amscd,%paralist
}

\numberwithin{equation}{section}

\newtheorem{thrm}{Theorem}[section]
\newtheorem{lemma}[thrm]{Lemma}
\newtheorem{prop}[thrm]{Proposition}
\newtheorem{cor}[thrm]{Corollary}
\newtheorem{dfn}[thrm]{Definition}

\newtheorem*{rmrk*}{Remark}

\newtheorem{conv}[thrm]{Convention}

\usepackage[margin=1in]{geometry}
\newcommand{\QH}{\boldsymbol {G\,(\mathbb{H})}}

\newcommand{\A}{{A}}
\newcommand{\B}{{B}}
\newcommand{\C}{{C}}
\newcommand{\D}{{D}}
\newcommand{\spl}{\mathfrak{ sp}(1)}

\newcommand{\lc}{\langle}
\newcommand{\rc}{\rangle}

\begin{document}

\begin{abstract}
 The main result is that the qc-scalar curvature of a seven dimensional quaternionic contact Einstein manifold is a constant. In addition, we characterize qc-Einstein structures with certain flat vertical connection and develop their local structure equations. Finally, regular qc-Ricci flat structures are shown to fiber over hyper-K\"ahler manifolds.
\end{abstract}

\title[Quaternionic contact  Einstein manifolds]
{Quaternionic contact  Einstein manifolds}
\date{\today}

\author{Stefan Ivanov}
\address[Stefan Ivanov]{University of Sofia, Faculty of Mathematics and Informatics,
blvd. James Bourchier 5, 1164,
Sofia, Bulgaria} \email{ivanovsp@fmi.uni-sofia.bg}

\author{Ivan Minchev}
\address[Ivan Minchev]{University of Sofia, Faculty of Mathematics and Informatics,
blvd. James Bourchier 5, 1164,
Sofia, Bulgaria}
\email{minchev@fmi.uni-sofia.bg}

\author{Dimiter Vassilev}
\address[Dimiter Vassilev]{ Department of Mathematics and Statistics\\
University of New Mexico\\
Albuquerque, New Mexico, 87131-0001}
\email{vassilev@math.unm.edu}

\maketitle

%\date{April 17, 2006}

\setcounter{tocdepth}{1} \tableofcontents
\section{Introduction}
Following the work of Biquard \cite{Biq1}  quaternionic contact (qc) manifolds describe the  Carnot-Carath\'eodory geometry  on the conformal boundary at infinity of quaternionic K\"ahler manifolds.  The qc geometry also became  a crucial geometric tool in finding the extremals and the best constant in the $L^2$ Folland-Stein Sobolev-type embedding on the quaternionic Heisenberg groups  \cite{F2,FS},  see  \cite{IMV,IMV2,IMV3}. An extensively studied class of  quaternionic contact structures are provided by the 3-Sasakian manifolds.   From the point of view of qc geometry, 3-Sasakian structures are qc manifolds whose torsion endomorphism of the Biquard connection vanishes. In turn, the latter property is equivalent to the qc structure being qc-Einstein, i.e., the trace-free part of the qc-Ricci tensor vanishes, see \cite{IMV}. The qc-scalar curvature of a 3-Sasakian manifold is a non-zero constant. Conversely, it was shown in \cite{IMV,IV2}  that the Biquard torsion is the obstruction for a given qc structure to be locally 3-Sasakian provided the qc-scalar curvature $Scal$ is a non zero constant. Furthermore,  as a consequence of the Bianchi identities,  \cite[Theorem 4.9]{IMV} shows that the qc-scalar curvature of a qc-Einstein manifold of dimension  at least eleven is constant while the seven dimensional case was left open.

 The main purpose of this paper is to show that the qc-scalar curvature of a seven dimensional qc-Einstein manifold is constant, i.e., to prove the following
\begin{thrm}\label{t:main}
If $M$ is a qc-Einstein qc manifold of dimension seven, then, the qc-scalar curvature is  a constant,  $Scal=const$.
\end{thrm}
The proof of  Theorem~\ref{t:main} makes use of the qc-conformal curvature tensor \cite{IV1}, which characterizes  locally qc conformally flat structures, a  result of Kulkarni \cite{Kul} on algebraic properties of curvature tensors in four dimensions,  and an extension of \cite[Theorem~1.21]{IMV}  which describes explicitly all  qc-Einstein structures defined on open sets of the quaternionic Heisenberg group that are  point-wise qc-conformal to the standard flat qc structure on the quaternionic Heisenberg groups. The main application of Theorem \ref{t:main} is the removal of the a-priori assumption of constancy of the qc-scalar curvature in some previous papers concerning  seven dimensional qc-Einstein manifolds, see for example Corollaries \ref{c:vert int}, \ref{c:4-form} and \ref{c:3-sasakian}.

The remaining parts of this paper are motivated by known properties of  qc-Einstein manifolds with non-vanishing qc-scalar curvature,   in that we prove corresponding results in the case of vanishing qc-scalar curvature.  With this goal in mind and because of its independent interest, in Section~\ref{three}  we define a  connection on the canonical three dimensional vertical distribution  of a qc manifold. We show  that   qc-Einstein spaces can be characterized by the  flatness of this vertical  connection. This allows us to write the structure equations of a qc-Einstein manifold in terms of the defining 1-forms, their exterior derivatives and the qc-scalar curvature, see Theorem~\ref{str_eq_mod_th}. The latter extends the results of  \cite{IV2} and \cite[Section 4.4.2]{IV3} to the vanishing qc-scalar curvature case.

Recall that complete and regular  3-Sasakian and $nS$-spaces
(called negative 3-Sasakian here) have canonical fibering with
fiber $Sp(1)$ or $SO(3)$, and base a quaternionic K\"ahler
manifold. The shows  that if $S>0$  (resp. $S<0$), the qc Einstein
manifolds are "essentially" $SO(3)$ bundles over quaternionic
K\"ahler manifolds with positive (resp. negative) scalar
curvature. In section \ref{five} we show that in the "regular"
case, similar to the non-zero qc-scalar curvature cases, a
qc-Einstein manifold of zero scalar curvature fibers over a
hyper-K\"ahler manifold, see Proposition~\ref{t:hKquot}.

We conclude the paper with a brief section where we show that
every qc-Einstein manifold of non-zero scalar curvature carries
two Einstein metrics. Note that the corresponding results
concerning the 3-Sasakian case is  well known, see \cite{BGN}.  In
the negative qc-scalar curvature case both Einstein metrics are of
signature $(4n,3)$  of which the first  is locally (negative)
3-Sasakian, while the second "squashed' metric is not  3-Sasakian,
see Proposition~\ref{p:einst m}.

\begin{conv}\label{conven}
Throughout the paper, unless explicitly stated otherwise, we will
use the following conventions.
\begin{enumerate}[a)]\label{e:notation}
\item The triple $(i,j,k)$ denotes any cyclic permutation of
$(1,2,3)$ while  $s,t$ will denote any numbers from the set $\{1,2,3\}$, $s,t \in \{1,2,3\}$.
\item For a  decomposition $TM=V\oplus H$ we let
$[.]_V$ and $[.]_H$be  the corresponding projections to $V$ and $H$.
\item ${\A}, {\B}, {\C}$, etc. will denote
sections of the tangent bundle of $M$, i.e., ${\A},
{\B}, {\C}\in  TM$.
\item $X,Y,Z,U$ will denote horizontal vector fields, i.e.,
$X,Y,Z,U\in H$.
\item $\xi,\xi',\xi''$ will denote vertical vector fields, i.e., $\xi,\xi',\xi''\in
V$.
\item $\{e_1,\dots,e_{4n}\}$ denotes an orthonormal basis of the
horizontal space $H$;
\item The summation convention over repeated vectors from the
basis $ \{e_1,\dots,e_{4n}\}$ is used. For example,
$k=P(e_b,e_a,e_a,e_b)$ means $
k=\sum_{a,b=1}^{4n}P(e_b,e_a,e_a,e_b). $
\end{enumerate}
\end{conv}

\textbf{Acknowledgments} The research is partially supported by
the Contract ``Idei", DID 02-39/21.12.2009. S.I and I.M. are
partially supported by the Contract 156/2013 with the University
of Sofia `St.Kl.Ohridski'.

\section{Preliminaries}\label{s:prelim}
It is well known that the sphere at infinity of a  non-compact symmetric space $M$ of rank one carries a natural
Carnot-Carath\'eodory structure, see \cite{M,P}.  Quaternionic contact (qc) structure were introduced by O. Biquard \cite{Biq1} and are modeled on  the conformal boundary at infinity of the quaternionic hyperbolic space. Biquard showed that the infinite dimensional family \cite{LeB91} of complete quaternionic-K\"ahler deformations of the quaternion hyperbolic metric have conformal infinities which provide an infinite dimensional family of examples of qc structures. Conversely, according to \cite{Biq1} every real analytic qc structure on a manifold $M$ of dimension at least eleven is the conformal infinity of a unique quaternionic-K\"ahler metric defined in a neighborhood of $M$. Furthermore,  \cite{Biq1} considered CR and qc structures as boundaries of infinity of Einstein metrics rather than only as boundaries at infinity of  K\"ahler-Einstein and quaternionic-K\"ahler metrics, respectively.  In fact, \cite{Biq1} showed that in each of the three hyperbolic cases (complex, quaternionic, octoninoic)  any small perturbation of the standard Carnot-Carath\'eodory structure on the boundary is the conformal infinity of an essentially unique Einstein metric on the unit ball, which is asymptotically symmetric.

 We refer to \cite{Biq1}, \cite{IMV} and \cite{IV3} for a more detailed exposition of the definitions and properties of qc structures and the associated Biquard connection. Here, we recall briefly the relevant facts needed for this paper. A quaternionic contact (qc) manifold is a $4n+3$%
-dimensional manifold  $M$ with a codimension three distribution $H$  equipped with  an $Sp(n)Sp(1)$ structure locally defined by an $\mathbb{R}^3$-valued 1-form $\eta=(\eta_1,\eta_2,\eta_3)$. Thus, $H=\cap_{s=1}^3 Ker\, \eta_s$
is equipped with a positive definite symmetric tensor $g$, called the horizontal metric, and a compatible rank-three bundle $\mathbb{Q}$
consisting of endomorphisms of $H$ locally generated by three orthogonal almost complex
structures $I_s$, $s=1,2,3$, satisfying the unit quaternion relations: (i) $I_1I_2=-I_2I_1=I_3, \quad $ $I_1I_2I_3=-id_{|_H}$; \hskip.1in (ii) $g(I_s.,I_s.)=g(.,.)$; and \hskip.1in  (iii) the
compatibility conditions  $2g(I_sX,Y)\ =\ d\eta_s(X,Y)$, $
X,Y\in H$  hold true.

The transformations preserving a given quaternionic contact
structure $\eta$, i.e., $\bar\eta=\mu\Psi\eta$ for a positive smooth
function $\mu$ and an $SO(3)$ matrix $\Psi$ with smooth functions as
entries are called \emph{quaternionic contact conformal (qc-conformal) transformations}.  If the function $\mu$ is constant $\bar\eta$ is called \emph{qc-homothetic} to $\eta$ and in the case $\mu\equiv 1$ we call $\bar\eta$ \emph{qc-equivalent} to $\eta$. Notice that in the latter case, $\eta$ and $\tilde\eta$ define the same qc structure. The qc conformal curvature tensor $W^{qc}$, introduced in \cite{IV1}, is the
obstruction for a qc structure to be locally qc conformal to the
standard 3-Sasakian structure on the $(4n+3)$-dimensional sphere \cite{IV1,IV3}.

Biquard showed that on a qc manifold of dimension at least eleven
there  is a unique connection $\nabla$ with torsion $T$  and a
unique supplementary to $H$ in $TM$  subspace $V$, called the\emph{ vertical space}, such that the following conditions are satisfied:
%\begin{enumerate}[(i)] \item
(i) $\nabla$ preserves the decomposition $H\oplus V$ and the $
Sp(n)Sp(1)$ structure on $H$, i.e.,  $\nabla g=0, \nabla\sigma \in\Gamma(
\mathbb{Q})$ for a section $\sigma\in\Gamma(\mathbb{Q})$, and its torsion on
$H$ is given by $T(X,Y)=-[X,Y]_{|V}$; \quad
%\item
(ii) for $\xi\in V$, the endomorphism $T_{\xi }=T(\xi ,\cdot ):H\rightarrow H$ of $H$ lies in $
(sp(n)\oplus sp(1))^{\bot}\subset gl(4n)$; \quad
%\item
(iii) the connection on $V$ is induced by the natural identification $
\varphi$ of $V$ with the subspace $sp(1)$ of the endomorphisms of $H$, i.e., $
\nabla\varphi=0$.
%\end{enumerate}
%\end{thrm}
Furthermore, \cite{Biq1} also described the supplementary distribution $V$,  which is (locally) generated by the so called Reeb vector fields $%
\{\xi_1,\xi_2,\xi_3\}$ determined by
\begin{equation}  \label{bi1}
 \eta_s(\xi_t)=\delta_{st}, \qquad (\xi_s\lrcorner
d\eta_s)_{|H}=0,\quad (\xi_s\lrcorner d\eta_t)_{|H}=-(\xi_t\lrcorner
d\eta_s)_{|H},
\end{equation}
where $\lrcorner$ denotes the interior multiplication.

If the dimension of $M $ is seven Duchemin showed in \cite{D} that if we assume, in addition, the
existence of Reeb vector fields as above, then   the Biquard result  holds. Henceforth, by a qc structure in dimension $7$ we shall mean a qc structure satisfying \eqref{bi1}.  We shall call $\nabla$ \emph{the Biquard connection}.

Notice that equations \eqref{bi1} are invariant under the natural $SO(3)$
action. Using the triple of Reeb vector fields we extend the horizontal  metric $g$  to a metric $h$ on $M$ by requiring $span\{\xi_1,\xi_2,\xi_3\}=V\perp H
\text{ and } h(\xi_s,\xi_t)=\delta_{st} $,
\begin{equation}\label{e:Riem-metric}
h|_H=g, \qquad h|_V= \eta_1\otimes\eta_1+ \eta_2\otimes\eta_2 + \eta_3\otimes\eta_3,\qquad h(\xi_s,X)=0.
\end{equation}
The Riemannian metric $h$ as well as the Biquard connection do not depend on the action of $SO(3)$ on $V$, but both change  if $\eta$ is multiplied by a conformal factor \cite{IMV}.

 The fundamental 2-forms $\omega_s$ and the { fundamental 4-form} $\Omega$  of the quaternionic structure $\mathbb{Q}$ are
defined, respectively, by
%\begin{equation}  \label{thirteen}
\[
2\omega_{s|H}\ =\ \, d\eta_{s|H},\quad \xi\lrcorner\omega_s=0,\qquad
\Omega=\omega_1\wedge\omega_1+\omega_2\wedge\omega_2+\omega_3\wedge\omega_3.
\]
%\end{equation}

\subsection{The torsion of the Biquard connection}
It was shown in \cite{Biq1} that the
torsion $T_{\xi }$ is completely trace-free, $tr\,T_{\xi }=tr\,T_{\xi }\circ
I_{s}=0$.  Decomposing the endomorphism $T_{\xi }\in (sp(n)+sp(1))^{\perp }$
into its symmetric part $T_{\xi }^{0}$ and skew-symmetric part $b_{\xi
},T_{\xi }=T_{\xi }^{0}+b_{\xi }$, we have  $T_{\xi
_{i}}^{0}I_{i}=-I_{i}T_{\xi _{i}}^{0}\quad I_{2}(T_{\xi
_{2}}^{0})^{+--}=I_{1}(T_{\xi _{1}}^{0})^{-+-},\quad I_{3}(T_{\xi
_{3}}^{0})^{-+-}=I_{2}(T_{\xi _{2}}^{0})^{--+},\quad I_{1}(T_{\xi
_{1}}^{0})^{--+}=I_{3}(T_{\xi _{3}}^{0})^{+--}$, where the upper script $+++$
denotes the component commuting with all three $I_{i}$, $+--$ indicates the component  commuting with $%
I_{1} $ and anti-commuting with the other two, etc. Furthermore, the symmetric part
$T_\xi^0$  satisfies the identity
\begin{equation}\label{to}
g(T_\xi^0(X),Y)=\frac{1}{2}\mathcal{L}_{\xi}g(X,Y),\qquad \xi\in
V,\quad X,Y\in H,
\end{equation}
where $\mathcal{L}_{\xi}$ denotes the Lie derivative with respect to $\xi$. The skew-symmetric
part can be represented as $b_{\xi _{i}}=I_{i}u$, where $u$ is a traceless
symmetric (1,1)-tensor on $H$ which commutes with $I_{1},I_{2},I_{3}$.
Therefore we have $T_{\xi _{i}}=T_{\xi _{i}}^{0}+I_{i}u$. If $n=1$ then the
tensor $u$ vanishes identically, $u=0$, and the torsion is a symmetric
tensor, $T_{\xi }=T_{\xi }^{0}$.
Following \cite{IMV} we define the $Sp(n)Sp(1)$ components  $T^0$ and $U$ of the torsion tensor by
\begin{gather*}% \label{tor}
T^0(X,Y)\ {=}\ g((T_{\xi_1}^{0}I_1+T_{\xi_2}^{0}I_2+T_{%
\xi_3}^{0}I_3)X,Y),\qquad
U(X,Y)\ {=}\ -g(uX,Y).
\end{gather*}
Then, as shown in \cite{IMV}, both $T^0$ and $U$ are trace-free,
symmetric and invariant under qc homothetic transformations.
Using the fixed horizontal metric $g$, we shall also denote by $T^0$ and $U$ the corresponding
endomorphisms of $H$, $g(T^0(X),Y)=T^0(X,Y)$ and $g(U(X),Y)=U(X,Y)$.  The torsion of the Biquard connection $\nabla$
is described by the formulas \cite{Biq1} and \cite{IMV}
\begin{equation}\label{torsion}
\begin{aligned}
& T(X,Y)  = -[X,Y]_V=2\sum_{s=1}^3\omega_s(X,Y)\xi_s,\qquad
 T(\xi_s,X) = \frac{1}{4}(I_sT^0-T^0I_s)(X)+I_sU(X),\\
& \hskip 1in T(\xi_i,\xi_j) = -S\xi_k-[\xi_i,\xi_j]_H,
\end{aligned}
\end{equation}
where $[\xi_i,\xi_j]_H$ stands for the $H$-component of the
commutator of the vector fields $\xi_i$, $\xi_j$ and $S$ is the normalized qc-scalar curvature defined below.

\subsection{The curvature of the Biquard connection}
We denote by
$R=[\nabla,\nabla]-\nabla_{[,]}$ the curvature tensor of $\nabla$ and by the same letter
$R$ the curvature $(0,4)$-tensor
$R(\A,\B,\C,\D):=h(R_{\A,\B}\C,\D).$ The \emph{qc-Ricci tensor},  the \emph{qc-scalar
curvature}, and the three \emph{qc-Ricci 2-forms} are defined as follows,
\begin{equation}\label{neww}
Ric(\A,\B)=R(e_a,\A,\B,e_a),\quad Scal=Ric(e_a,e_a),\quad
{\rho_s(\A,\B)=\frac{1}{4n}R(\A,\B,e_a,I_se_a)}.
\end{equation}
The \emph{normalized qc-scalar curvature} $S$ is defined by  $8n(n+2)S=Scal$.

A fundamental fact, \cite[Theorem 3.12]{IMV}, is that
the torsion endomorphism determines the (horizontal) qc-Ricci tensor
and the (horizontal) qc-Ricci forms of the Biquard connection,
\begin{equation}\label{sixtyfour}
\begin{aligned}
& Ric(X,Y) \ =\ (2n+2)T^0(X,Y) +(4n+10)U(X,Y)+2(n+2)Sg(X,Y)\\
 &\rho_s(X,I_sY) \  =\
-\frac12\Bigl[T^0(X,Y)+T^0(I_sX,I_sY)\Bigr]-2U(X,Y)-%
Sg(X,Y).
\end{aligned}
\end{equation}
We say that $M$ is a \emph{qc-Einstein manifold} if the horizontal Ricci tensor is proportional to the horizontal metric  $g$,
$$Ric(X,Y)=\frac{Scal}{4n}g(X,Y)=2(n+2)Sg(X,Y),$$
which taking into account  \eqref{sixtyfour}  is equivalent to $T^0=U=0$. Furthermore, by \cite[Theorem 4.9]{IMV} if $\dim(M)>7$ then any qc-Einstein structure has a constant qc-scalar curvature. It was left as an open question whether a qc-Einsten manifold of dimension seven has constant qc-scalar curvature. The main result of the current paper Theorem \ref{t:main} shows that this is indeed the case.

If the covariant derivatives with respect to $\nabla$ of the endomorphisms $I_s$, the fundamental 2-forms $\omega_s$, and the Reeb vector fields $\xi_s$  are given by
\begin{equation}\label{der}
\nabla I_i=-\alpha_j\otimes I_k+\alpha_k\otimes I_j,\qquad
\nabla\omega_i=-\alpha_j\otimes\omega_k+\alpha_k\otimes\omega_j,\qquad
%\end{equation}
%\begin{equation*}
\nabla\xi_i=-\alpha_j\otimes\xi_k+\alpha_k\otimes\xi_j,
\end{equation}
where $\alpha_1,\alpha_2, \alpha_3$ are the local connection 1-forms, then
\cite{Biq1} proved that $\alpha_i(X)=d\eta_k(\xi_j,X)=-d\eta_j(\xi_k,X)\quad \text{for all}\quad X\in H$. {On the other hand, as shown in \cite{IMV} the vertical and the $\spl$
parts of the curvature endomorphism $R(\A,\B)$ are related to the $\spl$-connection 1-forms
$\alpha_s$ by }
\begin{equation}  \label{sp1curv}
R(\A,\B,\xi_i,\xi_j)=2\rho_k(\A,\B)=(d\alpha_k+\alpha_i\wedge\alpha_j)(\A,\B).
\end{equation}
Finally, we  have the following commutation relations \cite{IMV}
\begin{equation}\label{rrho}
R(B,C,I_iX,Y)+R(B,C,X,I_iY)=2\Big[-\rho_j(B,C)\omega_k(X,Y)+\rho_k(B,C)\omega_j(X,Y)\Big].
\end{equation}

In the next section we give the proof of our main result.

\section{Proof of Theorem~\ref{t:main}}

The proof of  Theorem~\ref{t:main} is achieved with the help of the following  Lemma \ref{lemman} where we calculate the curvature $R(Z,X,Y,V)$ at points where the horizontal gradient of the qc-scalar curvature does not vanish,  $\nabla S\not=0$.  The proof of Theorem~\ref{t:main} proceeds by showing  that on any open set where $S$ is not locally constant $M$ is locally qc-conformally flat. In fact, on any open set where $\nabla S\not=0$ the qc-conformal curvature $W^{qc}$ defined in \cite{IV1} will be seen to vanish, hence by \cite[Theorem~1.2]{IV1}  the qc manifold is locally qc-conformally flat. The final step involves a  generalization of \cite [Theorem~1.1]{IMV}, which follows from the proof of  \cite [Theorem~1.1]{IMV}, allowing the explicit description of all  qc-Einstein structures defined on open sets of the quaternionic Heisenberg group that are  point-wise qc-conformal to the standard flat qc structure on the quaternionic Heisenberg groups. It turns out that all such qc structures are of constant qc-scalar curvature, which allows the completion of the proof of Theorem \ref{t:main}.

\begin{lemma}\label{lemman}
On a seven dimensional qc-Einstein  manifold we have the following formula for the horizontal curvature of the Biquard connection on any open set where the qc-scalar curvature is not constant,
\begin{equation}\label{n17}
R(Z,X,Y,V)=2S\Big[g(Z,V)g(X,Y)-g(X,V)g(Z,Y)\Big].
\end{equation}
\end{lemma}
\begin{proof}[Proof of Lemma \ref{lemman}]
Our first goal is to show the next identity,
\begin{equation}\label{n15}
R(Z,X,Y,\nabla S)=2S\Big[dS(Z)g(X,Y)-dS(X)g(Z,Y)\Big],
\end{equation}
where $\nabla S$ is the horizontal gradient of $S$ defined by $g(X,\nabla S)=dS(X)$.  For this, recall the general formula proven in \cite[Theorem 3.1, (3.6)]{IV1},
\begin{multline}\label{vert2}
 R(\xi_i,\xi_j,X,Y)=(\nabla_{\xi_i}U)(I_jX,Y)-(\nabla_{\xi_j}U)(I_iX,Y)\\
 -\frac14\Big[(\nabla_{\xi_i}T^0)(I_jX,Y)+(\nabla_{\xi_i}T^0)(X,I_jY)\Big]
 +\frac14\Big[(\nabla_{\xi_j}T^0)(I_iX,Y)+(\nabla_{\xi_j}T^0)(X,I_iY)\Big]\\
 -(\nabla_X\rho_k)(I_iY,\xi_i) -\frac{Scal}{8n(n+2)}T(\xi_k,X,Y)
 -T(\xi_j,X,e_a)T(\xi_i,e_a,Y)+T(\xi_j,e_a,Y)T(\xi_i,X,e_a)
\end{multline}
where the Ricci two forms are given by, cf. \cite[Theorem 3.1]{IV1},
\begin{equation}
\begin{aligned}
&  6(2n+1)\rho_s(\xi_s,X)=(2n+1)X(S)+\frac12\Big[(\nabla_{e_a}T^0)(e_a,X)-3(\nabla_{e_a}T^0)(I_se_a,I_sX)\Big]
-2(\nabla_{e_a}U)(e_a,X),\\
&  6(2n+1)\rho_i(\xi_j,I_kX)=-6(2n+1)\rho_i(\xi_k,I_jX)=(2n-1)(2n+1)X(S)\hfill\\
& \hskip1in -\frac12\Big [ (4n+1)(\nabla_{e_a}T^0)(e_a,X)+3(\nabla_{e_a}T^0)(I_ie_a,I_iX)\Big]-4(n+1)(%
\nabla_{e_a}U)(e_a,X) .
\hfill\end{aligned}  \label{d3n6}
\end{equation}
In our case  $T^0=U=0$, hence  \eqref{vert2} takes the form
\begin{equation}\label{n2}
R(\xi_i,\xi_j,X,Y)=-(\nabla_X\rho_k)(I_iY,\xi_i).
\end{equation}
Letting $n=1$ and $T^0=U=0$ in \eqref{d3n6} it follows $\rho_i(I_kY,\xi_j)=-\frac16dS(Y)$, which after a cyclic permutation of $ijk$ and a substitution of  $Y$ with $I_kY$ yields
\begin{equation}\label{n1}
\rho_k(I_iY,\xi_i)=-\frac16dS(I_kY).
\end{equation}
Taking the covariant derivative of \eqref{n1} with respect to the Biquard connection and applying \eqref{der} we calculate
\begin{multline}\label{n3}
(\nabla_X\rho_k)(I_iY,\xi_i)-\alpha_i(X)\rho_j(I_iY,\xi_i)+\alpha_j(X)\rho_i(I_iY,\xi_i)-\alpha_j(X)\rho_k(I_kY,\xi_i)+\alpha_k(X)\rho_k(I_jY,\xi_i)\\-\alpha_j(X)\rho_k(I_iY,\xi_k)+\alpha_k(X)\rho_k(I_iY,\xi_j)=-\frac16\nabla^2S(X,I_kY)+\frac16\alpha_i(X)dS(I_jY)-\frac16\alpha_j(X)dS(I_iY).
\end{multline}
Applying \eqref{d3n6} with $n=1$ and $T^0=U=0$ we see that the terms involving the connection 1-forms cancel and \eqref{n3} turns into
\begin{equation}\label{n4}
(\nabla_X\rho_k)(I_iY,\xi_i)=-\frac16\nabla^2S(X,I_kY).
\end{equation}
A substitution of \eqref{n4} in \eqref{n2} taking into account the skew-symmetry of $R(\xi_i,\xi_j,X,Y)$ with respect to $X$ and $Y$ allows us to conclude the following identity for the horizontal Hession of $S$
\begin{equation}\label{n5}
\nabla^2S(X,I_sY)+\nabla^2S(Y,I_sX)=0.
\end{equation}
The trace of \eqref{n5} together with the Ricci identity yield
\begin{multline*}%\label{n6}
0=2\nabla^2S(e_a,I_ke_a)=\nabla^2S(e_a,I_ke_a)-\nabla^2S(I_ke_a,e_a)
=-2\sum_{s=1}^3\omega_s(e_a,I_ke_a)dS(\xi_s)=-8dS(\xi_k),
\end{multline*}
i.e., we have
\begin{equation}\label{n6}
dS(\xi_s)=0, \qquad \nabla^2S(\xi_s,\xi_t)=0.
\end{equation}
The equality \eqref{n6} shows that $S$ is constant along the vertical directions, $dS(\xi_s)=0$, hence, in view of \eqref{der}, the second equation of \eqref{n6} holds as well. In addition, we have
$\nabla^2S(X,\xi_s)=XdS(\xi_s)-dS(\nabla_X\xi_s)=0$ since $\nabla$ preserves the vertical directions due to \eqref{der}.
Moreover, the Ricci identity $$\nabla^2S(\xi_s,X)-\nabla^2S(X,\xi_s)=dS(T(\xi_s,X))=0$$ together with the above equality leads to
\begin{equation}\label{n7}
\nabla^2S(\xi_s,X)=\nabla^2S(X,\xi_s)=0.
\end{equation}

Next, we show that the horizontal Hessian of $S$ is symmetric. Indeed, we have the identity
\begin{equation}\label{n8}
\nabla^2S(X,Y)-\nabla^2S(Y,X)=d^2S(X,Y)-dS(T(X,Y)) =-2\sum_{s=1}^3\omega_s(X,Y)dS(\xi_s)=0
\end{equation}
where we applied \eqref{n6} to conclude the last equality.  Now, \eqref{n5} and \eqref{n8} imply
\begin{equation}\label{n9a}
\nabla^2S(X,Y)-\nabla^2S(I_sX,I_sY)=0
\end{equation}
which shows that the $[-1]$-component of the horizontal Hessian vanishes. Hence, the horizontal Hessian of $S$ is proportional to the horizontal metric since $n=1$, i.e.,
\begin{equation}\label{n9}
\nabla^2S(X,Y)=\frac{\nabla^2S(e_a,e_a)}4g(X,Y)=-\frac{\triangle S}4g(X,Y),
\end{equation}
where $\triangle S=-\nabla^2S(e_a,e_a)$ is the sub-Laplacian of $S$.  We have the following Ricci identity of order three (see e.g. \cite{IPV}
\begin{equation}\label{n10}
\nabla^3 S(X,Y,Z)-\nabla^3 S(Y,X,Z)=-R(X,Y,Z,\nabla S) - 2\sum_{s=1}^3
\omega_s(X,Y)\nabla^2S (\xi_s,Z).
\end{equation}
Applying \eqref{n7} we conclude from \eqref{n10} that
\begin{equation}\label{n11}
\nabla^3 S(X,Y,Z)-\nabla^3 S(Y,X,Z)=-R(X,Y,Z,\nabla S).
\end{equation}
Combining \eqref{n11} and \eqref{n9} we obtain  the next expression for  the curvature
\begin{equation}\label{n12}
R(Z,X,Y,\nabla S)=\frac{\nabla^3S(X,e_a,e_a)}4g(Z,Y)-\frac{\nabla^3S(Z,e_a,e_a)}4g(X,Y).
\end{equation}
The trace of \eqref{n12} together with the first equality of \eqref{sixtyfour} computed for $n=1, T^0=0$ and $U=0$ yield
$$Ric(Z,\nabla S)=6SdS(Z)=-\frac34\nabla^3S(Z,e_a,e_a).$$
Thus, we have
\begin{equation}\label{n14}
\nabla^3S(Z,e_a,e_a)=-8SdS(Z).
\end{equation}
Now, a substitution  of \eqref{n14} in \eqref{n12} gives \eqref{n15}.

Turning to the general formula \eqref{n17} we note that the horizontal curvature of the Biquard connection in the  qc-Einstein case
satisfies the identity
\begin{equation}\label{b1}R(X,Y,Z,V)+R(Y,Z,X,V)+R(Z,X,Y,V)=0.
\end{equation}
This follows from the first Bianchi identity since $(\nabla T)(X,Y)=0$ and $ T(T(X,Y),Z)=\sum_{s=1}^32\omega_s(X,Y)T(\xi_s,Z)=0.$
Thus,  the horizontal curvature has the algebraic properties of  the Riemannian curvature, namely it is skew-symmetric with respect to the first and the last pairs and satisfies the Bianchi identity \eqref{b1}. Therefore it also has the fourth Riemannian curvature property,
\begin{equation}\label{riem}
R(X,Y,Z,V)=R(Z,V,X,Y).
\end{equation}
The equalities \eqref{n15} and \eqref{riem} imply
\begin{gather}\label{n16}
0=R(I_i\nabla S,I_j\nabla S,I_k\nabla S,\nabla S)=R(I_k\nabla S,\nabla S,I_i\nabla S,I_j\nabla S), \\\nonumber 0= R(I_i\nabla S,I_j\nabla S,I_j\nabla S,\nabla S)= R(I_j\nabla S,\nabla S,I_i\nabla S,I_j\nabla S).
\end{gather}
Moreover, using \eqref{rrho} and the second equality in \eqref{sixtyfour} with $T^0=U=0$  we calculate
\begin{multline}\label{scur1}
R(I_j\nabla S,I_i\nabla S,I_i\nabla S,I_k\nabla S)-R(I_j\nabla S,I_i\nabla S,\nabla S,I_j\nabla S)\\=-2\rho_j(I_j\nabla S,I_i\nabla S)\omega_k(\nabla S,I_k\nabla S)+2\rho_k(I_j\nabla S,I_i\nabla S)\omega_j(\nabla S,I_k\nabla S)=0
\end{multline}
The second equality of \eqref{n16} together with \eqref{scur1} yields
\begin{equation}\label{scur2}
R(I_j\nabla S,I_i\nabla S,I_i\nabla S,I_k\nabla S)=0.
\end{equation}
Finally, \eqref{n15}, \eqref{n16}, \eqref{scur1}, \eqref{scur2} together with \eqref{rrho} and \eqref{sixtyfour}  imply for any $s\not=t$ the identities
\begin{equation}\label{scur3}
R(I_s\nabla S,I_t\nabla S,I_t\nabla S,I_s\nabla S)=R(I_s\nabla S,\nabla S,\nabla S,I_s\nabla S)=2S|\nabla S|^4.
\end{equation}
In a neighborhood of any point where $\nabla S\not=0$ the quadruple $\{\frac{\nabla S}{|\nabla S|}, \frac{I_1\nabla S}{|\nabla S|},\frac{I_2\nabla S}{|\nabla S|},\frac{I_3\nabla S}{|\nabla S|}\}$ is an orthonormal  basis of $H$, hence
%From \eqref{n16} it follows that the horizontal curvature vanishes on this basis of $H$.
 after a small calculation taking into account \eqref{n16}, \eqref{scur2} and \eqref{scur3}, we see that for any orthonormal basis  $\{Z,X,Y,V\}$ of $H$ we have
 \begin{equation}\label{kul1}
 R(Z,X,Y,V)=0, \qquad R(Z,X,Z,V)-R(Y,X,Y,V)=0,
 \end{equation}
where the second equation follows from the first using the orthogonal basis $\{Z+Y,X,Z-Y,V\}$.
For  the "sectional curvature" $K(Z,X)=R(Z,X,Z,X)$ we have then the identities
\begin{multline*}
K(Z,X)+K(Y,V)-K(Z,V)-K(Y,X)=R(Z,X,Z,X)+R(Y,V,Y,V) -R(Z,V,Z,V) - R(Y,X,Y,X)\\
=R(Y,X,Y,X)+R(Y,X,Y,V)-R(Y,V,Y,X)+R(Y,V,Y,V)-R(Z,X,Z,X)-R(Z,X,Z,V)\\
+R(Z,V,Z,X)-R(Z,V,Z,V)=R(Z,X+V,Z,X-V)-R(Y,X+V,Y,X-V)=0
\end{multline*}
using \eqref{riem} in the  second equality and \eqref{kul1} in the last equality. Now, \cite[Theorem 3]{Kul}, shows that the Riemannian conformal tensor of the horizontal curvature $R$ vanishes. In view of  $Ric=6S\cdot g$, we conclude that the curvature restricted to the horizontal space  is given by \eqref{n17} which proves the lemma.
\end{proof}

\begin{proof}[Proof of Theorem~\ref{t:main}]
Let $M$ be a qc-Einstein  manifold of dimension seven with a local $\mathbb{R}^3$-valued 1-form $\eta=(\eta_1,\eta_2,\eta_3)$ defining the given qc structure. Suppose the qc-scalar curvature is not a locally constant function. We shall reach a contradiction by showing that $M$ is locally qc conformally flat, which will be shown to imply that the qc-scalar curvature is locally constant.

To prove the first claim we prove  that if the qc-scalar curvature is not locally constant  then the qc-conformal curvature $W^{qc}$ of \cite{IV1} vanishes on the open set where $\nabla S\not=0$. For this we recall the formula for the qc-conformal curvature $W^{qc}$ given in \cite[Prposition~4.2]{IV1} which with the assumptions  $T^0=U=0$ simplifies to
\begin{multline}\label{qc1}
W^{qc}(Z,X,Y,V)=\frac14\Big[R(Z,X,Y,V)+\sum_{s=1}^3R(I_sZ,I_sX,Y,V)\Big]\\
+\frac{S}2\Big[g(Z,Y)g(X,V)-g(Z,V)g(X,Y)+\sum_{s=1}^3\Big(\omega_s(Z,Y)\omega_s(X,V)-
\omega_s(Z,V)\omega_s(X,Y)\Big)\Big].
\end{multline}
A substitution of \eqref{n17} in \eqref{qc1} shows $W^{qc}=0$ on $\nabla S\not=0$.

Now, \cite[Theorem~1.2]{IV1}  shows that the open set $\nabla S\not=0$ is locally qc-conformaly flat, i.e., every point $p$, $\nabla S(p)\not=0$ has an open neighborhood $O$ and a qc-conformal transformation $F: O \rightarrow \QH$ to the   quaternionic Heisenberg group $\QH$ equipped with the standard flat qc structure $\tilde\Theta$. Thus, $\Theta\overset{def}{=}{F}^*\eta= \frac{1}{2\mu} \tilde\Theta$ for some positive smooth function $\mu$ defined on the open set $F(O)$.  By its definition $\Theta$ is a qc-Einstein structure,  hence the proof of \cite[Theorem 1.1]{IMV} shows that,  with a small change of the parameters in \cite[Theorem 1.1]{IMV}, $\mu$ is given by
\begin{equation}\label{e:Liouville conf factor}
\mu (q,\omega) \ =\ c_0\ \Big [  \big ( \sigma\ +\
 |q+q_0|^2 \big )^2\  +\  |\omega\ +\ \omega_o\ +
\ 2\ \text {Im}\  q_o\, \bar q|^2 \Big ],
\end{equation}
for some fixed $(q_o,\omega_o)\in \QH$ and constants $c_0>0$ and $\sigma\in \mathbb{R}$. A small calculation using \eqref{e:Liouville conf factor} and the Yamabe equation  \cite[(5.8)]{IMV} shows $Scal_{\Theta}=128n(n+2)c_0\sigma =const$. Since $\eta$ is qc-conformal  to $\Theta$ via the map $F$,  it follows that  $Scal_{\eta}=const$ on $O$, which is a contradiction.
\end{proof}

An immediate consequence of Theorem \ref{t:main} and \cite[Theorem 4.9]{IMV} is the next
\begin{cor}\label{c:vert int}
The vertical space $V$ of a seven dimensional qc-Einstein manifold is integrable.
\end{cor}
We note that the integrability of the vertical distribution of a $4n+3$ dimensional qc-Einstein manifold in the case $n>1$, and when $S=const$ and $n=1$ was proven earlier in \cite[Theorem 4.9]{IMV}.   {Thus, in any dimension,  the vertical distribution $V$  of a qc-Einstein manifold is
integrable and we have}
\begin{equation}\label{e: some ricci of einstein}
 \rho_{s}(X,Y)=-S
\omega_s(X,Y), \qquad
 Ric(\xi_s,X)=\rho_s(X,\xi_t)=0, \qquad [\xi_s,\xi_t]\in V.
\end{equation}

 Another Corollary of Theorem \ref{t:main} and  the analysis of the corresponding results  in the case $n>1$\cite{IV2} is
\begin{cor}\label{c:4-form}
If $M$ is a seven dimensional  qc-Einstein manifold then $ d\Omega=0$, where $\Omega$ is the fundamental 4-form defining the quaternionic structure on the horizontal distribution.
\end{cor}
For  details, we refer to the proof of the case $n>1$ in \cite[Theorem 4.4.2. c)]{IV3} which  is valid in the case $n=1$, as well, due to  Theorem \ref{t:main} and Corollary \ref{c:vert int}.
We note that the converse to Corollary \ref{c:4-form} holds true when  $n>1$ , see  \cite{IV2}, while in  the case $n=1$ a counterexample for the implication was found in \cite{CFS}.

\addtocontents{toc}{\protect\setcounter{tocdepth}{2}}

\section{A characterization based on vertical flat connection}\label{three}
In  this section we  show that for any qc manifold $M$ there is a natural linear connection $\tilde \nabla$, defined on the vertical distribution $V$, the latter considered as a vector bundle over $M$. This connection has the remarkable property of being flat exactly when $M$ is qc-Einstein, see Theorem~\ref{flat tilde}, and  will turn out to be a useful technical tool for the geometry of qc Einstein manifolds in the sequel.

We start by introducing a cross-product on the vertical space $V$. Recall that  $h$  \eqref{e:Riem-metric} is the natural extension of the  horizontal metric $g$ to a Riemannian metric on $M$, which induces an inner product, denoted by $\lc.,.\rc$ here, and an orientation on the vertical distribution  $V$. This allows us to introduce also the
cross-product operation $\times:\Lambda^2(V)\rightarrow V$ in the
standard way: $\xi_i\times\xi_j=\xi_k,\ \xi_i\times\xi_i=0$.
The cross product operation is parallel with respect to
any connection on $V$ preserving the inner product $\lc.,.\rc$, in particular,  with respect to the
restriction of the Biquard connection $\nabla$ to $V$. For any  $\xi,\xi',\xi''\in V$, we have the standard relations
\begin{equation}\label{nablax}
\begin{aligned}
%\label{x.x}
&(\xi\times \xi')\times \xi''=\lc \xi,\xi'' \rc \xi' - \lc \xi',\xi'' \rc \xi, \qquad
%\label{x.x.x}
\xi\times (\xi'\times \xi'')=(\xi\times \xi')\times \xi''+\xi'\times (\xi\times \xi''),\\
& \hskip1in \nabla_{\A}(\xi\times \xi')=(\nabla_{\A}\xi)\times
\xi'+\xi\times (\nabla_{\A}\xi') .
\end{aligned}
\end{equation}

In the next lemma we collect some formulas, which will be used in the proof of Theorem~\ref{flat tilde}.
\begin{lemma}\label{Einstein:RT}
{The curvature $R$ and torsion $T$ of the Biquard connection $\nabla$ of a
qc-Einstein manifold satisfy the following
identities}
\begin{equation}\label{thri}
T(\xi,\xi')=-{S}\xi\times \xi',\quad T(\xi,X)=0, \quad R({\A},{\B})\xi=-2{
S}\sum_{s=1}^{3}\omega_s({\A},{\B})\xi_s\times \xi.
\end{equation}
\end{lemma}

\begin{proof}
{The first two identities follow directly from \eqref{torsion} and the integrability of the vertical distribution $V$, see Corollary~\ref{c:vert int} and the paragraph after it.}  The last identity follows from  \eqref{vert2}, \eqref{neww} and  \eqref{e: some ricci of einstein}. In particular,  the three Ricci 2-forms $\rho_s(\A,\B)$ vanish unless $\A$ and $\B$ are both
horizontal, in which case we have \eqref{e: some ricci of einstein}. The proof is complete.
\end{proof}

\begin{dfn}
We define a connection $ \widetilde{\nabla}$  on
the vertical vector bundle $V$ of a  qc manifold $M$ as follows
\begin{equation}\label{tildenabla}
 \widetilde{\nabla}_X\xi:=\nabla_X\xi,\qquad
  \widetilde{\nabla}_\xi\xi':=\nabla_\xi\xi'+{
S}(\xi\times \xi').
\end{equation}
\end{dfn}
The main result of this section is
\begin{thrm}\label{flat tilde}
A qc manifold $M$  is qc-Einstein iff the connection $\tilde \nabla$ is flat, $R^{\widetilde{\nabla}}=0$.
\end{thrm}
\begin{proof}We start by relating the curvature $R^{\widetilde{\nabla}}$ of the connection $\widetilde{\nabla}$, cf.  \eqref{tildenabla}, to the curvature of the Biquard connection $\nabla$. To this end,
let $L\ =\ (\widetilde{\nabla}-\nabla)\quad \in\ \Gamma(M,T^*M\otimes V^*\otimes
V)$ be the difference between the two connections on~$V$. Then \eqref{tildenabla} implies
$L_{\A}\xi=L({\A},\xi)={S}[{\A}]_V\times \xi$, where $[{\A}]_V$ is the orthogonal projection of $A$ on $V$.
The curvature tensor $R^{\widetilde{\nabla}}$ of the new connection $\widetilde{\nabla}$ is given in terms of $R$ and $L$ by the well known general formula
\begin{equation}\label{R-R}
R^{\widetilde{\nabla}}({\A},{\B})\xi=R({\A},{\B})\xi+\big(\nabla_{\A}L\big)({\B},\xi)-\big(\nabla_{\B}L\big)({\A},\xi) + \big[L_{\A},L_{\B}\big]\xi+L\big(T({\A},{\B}),\xi\big).
\end{equation}
We proceed by considering each of the terms on the right hand side of \eqref{R-R} separately. We have, cf. \eqref{sp1curv},
\begin{equation}\label{pR-R4}
R(\A,\B)\xi=\left(\sum_{s=1}^3 2\rho_s(\A,\B)\xi_s\right)\times \xi.
\end{equation}
Using \eqref{nablax} and the obvious identity $\nabla_\A
\big(\,[\B]_V\big)=\big[\nabla_\A\B\big]_V$ we obtain
\begin{equation}\label{pR-R1}
\big(\nabla_\A L\big)(\B,\xi)=\nabla_\A
\big(L(\B,\xi)\big)-L\big(\nabla_\A\B,\xi\big)-L\big(\B,\nabla_\A \xi\big)=dS(\A)[{\B}]_V\times
\xi.
\end{equation}
From \eqref{nablax} it follows
\begin{equation}\label{pR-R2}
\big[L_{\A},L_{\B}\big]\xi=\big(L_{\A}\times L_{\B}\big)\times \xi={S}^2\Big(\,[{\A}]_V\times [{\B}]_V\,\Big)\times \xi.
\end{equation}

\noindent The torsion identities \eqref{torsion} imply
\begin{equation}\label{pR-R3}
L\big(T(\A,\B\big),\xi)=S\big[T(\A,\B)\big]_V\times
\xi=S\Big(\,-S[\A]_V\times[\B]_V+2\sum_{s=1}^3\omega_s(\A,\B)\xi_s\,\Big)\times
\xi.
\end{equation}
Finally, a substitution of  \eqref{pR-R4}, \eqref{pR-R1}, \eqref{pR-R2} and \eqref{pR-R3} in the right hand side of formula \eqref{R-R} gives the equivalent relation
\begin{gather}\label{R_tilda}
R^{\widetilde{\nabla}}({\A},{\B})\xi=\left(\sum_{s=1}^3 2\rho_s(\A,\B)\xi_s+dS(\A)[{\B}]_V-dS(\B)[{\A}]_V+2S\sum_{s=1}^3\omega_s(\A,\B)\xi_s\right)\times \xi.
\end{gather}
We are now ready to complete the proof of the theorem. Suppose first that $M$ is a qc-Einstein manifold. By Theorem \ref{t:main} when $n=1$ and \cite{IMV} when $n>1$ it follows that the  qc-scalar curvature is constant. Lemma \ref{Einstein:RT} implies that
$$\sum_{s=1}^3\rho_s(\A,\B)\xi_s=-{ S}\sum_{s=1}^3\omega_s({\A},{\B})\xi_s.$$ Since $dS=0$,   \eqref{R_tilda} gives $R^{\widetilde{\nabla}}=0$, and thus $\widetilde{\nabla}$ is  a flat connection on $V$.

Conversely,  if $\widetilde{\nabla}$ is flat, then
by applying \eqref{R_tilda} with $(\A,\B)=(X,Y)$ we obtain $\rho_s(X,Y)=-{S}\omega_s(X,Y)$. Applying the second formula of \eqref{sixtyfour} we derive  $T^0=0$ and $U=0$
by comparing the $Sp(n)Sp(1)$  components of the obtained  equalities.  Thus, $(M,\eta)$ is a qc Einstein manifold taking into account the first formula in \eqref{sixtyfour}.
\end{proof}

\section{The structure equations of a qc Einstein manifold}

%\subsection{The structure equaions of a qc manifold}\label{ss:str equations}
Let $M$ be a qc manifold with normalized qc-scalar curvature $S$. From \cite[Proposition 3.1]{IV2} we have the structure equations
\begin{equation}\label{streq}
\begin{aligned}
d\eta_i & =2\omega_i-\eta_j\wedge\alpha_k+\eta_k\wedge\alpha_j -
S
\eta_j\wedge\eta_k,\\
d\omega_i & =\omega_j\wedge(\alpha_k+
S\eta_k)-\omega_k\wedge(\alpha_j+ S\eta_j)-\rho_k\wedge\eta_j+
\rho_j\wedge\eta_k+ \frac{1}{2}dS\wedge\eta_j\wedge\eta_k,%\label{strom}.
\end{aligned}
\end{equation}
  where $(\eta_1,\eta_2,\eta_3)$ is a local $\mathbb{R}^3$-valued 1-form defining the given qc-structure and $\alpha_s$ are the corresponding connection 1-forms.  If, {locally, there is an $\mathbb{R}^3$-valued 1-form $\eta=(\eta_1,\eta_2,\eta_3)$ defining the given qc-stricture, such that,} we have the structure equations $d\eta_i=2\omega_i+S\eta_j\wedge \eta_k$ with $S=const$  or  the connection 1-forms vanish on the horizontal space, ${\alpha_i}\vert_{H}=0$, then $M$ is a qc-Einstein manifold of normalized qc-scalar curvature $S$, see \cite[Proposition 3.1]{IV2} and   \cite[Lemma 4.18]{IMV}.

  Conversely, on a qc-Einstein manifold of nowhere vanishing qc-scalar curvature the structure equations \eqref{str_eq_mod} hold true by \cite{IV2} and \cite[Section 4.4.2]{IV3}, taking into account  Corollary \ref{c:3-sasakian}.  The purpose of this section is to give the corresponding results in the case $Scal=0$.  The proof of Theorem \ref{str_eq_mod_th} which is based on the connection defined in  Section \ref{three} rather than the cone over a 3-Sasakian manifold employed in \cite{IV2} and \cite[Theorem 4.4.4]{IV3} works also in the case $Scal\not=0$, thus in the statement of the Theorem we will not make an explicit note of the condition $Scal=0$.
\begin{thrm}\label{str_eq_mod_th} Let  $M$ be a qc manifold. The following conditions are equivalent:
\begin{enumerate}[a)]
\item $M$ is a qc Einstein manifold;
\item  locally, the given qc-structure is defined by  1-form $(\eta_1,\eta_2,\eta_3)$  such that for some constant $S$ we have
\begin{equation}\label{str_eq_mod}
d\eta_i=2\omega_i+S\eta_j\wedge\eta_k;
\end{equation}
\item locally,  the given qc-structure is defined by  1-form $(\eta_1,\eta_2,\eta_3)$  such that the corresponding connection 1-forms vanish on $H$, $\alpha_s=-S\eta_s.$
\end{enumerate}
\end{thrm}

\begin{proof}
 As explained above, the implication c) $\Rightarrow$  a) is  known, while b) $\Rightarrow$  c) is an immediate consequence of \eqref{streq}. Thus, only  the implication a) implies b) needs to be proven,  see also the paragraph preceding the Theorem.

Assume a) holds. We will show that the structure equation in b) are satisfied. By Theorem \ref{t:main} when $n=1$ and \cite{IMV} when $n>1$ it follows $M$ is of constant qc-scalar curvature.  Let $V$ be
the vertical distribution. Clearly, the
connection $\widetilde{\nabla}$ defined in Theorem~\ref{flat tilde} is a flat metric connection on $V$ with respect to the inner product
$\lc.,.\rc$. Therefore the bundle $V$ admits a local  orthonormal oriented frame $K_1,K_2,K_3$
which is $\widetilde{\nabla}$-parallel, i.e., we have
\begin{equation}\label{nabla K_i}
\nabla_{\A}K_i=-S[\A]_V\times K_i.
\end{equation}

\noindent There exists a  triple of local 1-forms  $(\eta_1,\eta_2,\eta_3)$ on $M$ vanishing on $H$, which satisfy
$\eta_s(K_t)=\delta_{st}$. We rewrite \eqref{nabla K_i} as
\begin{equation}\label{nabla K_i2}
\nabla_{\A}K_i=S\big(\eta_j(\A)K_k-\eta_k(\A)K_j\big).
\end{equation}
Since $K_1,K_2,K_3$ is an orthonormal and oriented frame of $V$, we can complete the dual triple $(\eta_1,\eta_2,\eta_3)$  to one defining the given qc-structure.  By differentiating the equalities $\eta_s(K_i)=\delta_{si}$ we obtain using \eqref{nabla K_i2} that
\begin{multline*}
0\ =\ \big(\nabla_{\A}\eta_s\big)(K_i)+\eta_s\big(\nabla_{\A}K_i\big)\ =\ \big(\nabla_{\A}\eta_s\big)(K_i)
+\eta_s\Big( S\big(\eta_j(\A)K_k-\eta_k(\A)K_j\big)\Big)\\
=\ \big(\nabla_{\A}\eta_s\big)(K_i)+S\Big(\eta_j(\A)\delta_{sk}-\eta_k(\A)\delta_{sj}\Big).
\end{multline*}
Hence,
%\begin{equation}\label{nabla K_i3}
$\big(\nabla_{\A}\eta_i\big)(\B)\ =\ S\eta_j\wedge\eta_k(\A,\B)$,
%\end{equation}
which together with Lemma~\ref{Einstein:RT} allows  the computation of
the exterior derivative of $\eta_i$,
\begin{multline}\label{deta Ki}
 d\eta_i(\A,\B)  =  \big(\nabla_\A\eta_i\big)(\B)-\big(\nabla_\B\eta_i\big)(\A)
+\eta_i\big(T(\A,\B)\big)=S\eta_j\wedge\eta_k(\A,\B)-S\eta_j\wedge\eta_k(\B,\A)\\
 + \eta_i\Big(-S[\A]_V\times[\B]_V +2\sum_s\omega_s(\A,\B)\xi_s\Big)
=\Big(2\omega_i+S\eta_j\wedge\eta_k\Big)(\A,\B),
\end{multline}
which proves
\eqref{str_eq_mod}.  Now $\alpha_s|_H=0$ shows that $K_s$ satisfy \eqref{bi1} and therefore $K_s$ are the Reeb vector fields, which completes the  proof of the Theorem.
\end{proof}
We finish the section with another condition characterizing qc-Einstein manifolds, which is useful in some calculations.
\begin{prop} Let  $M$ be a qc manifold. $M$ is qc-Einstein iff for some $\eta$ compatible with the given qc-structure
\begin{equation}\label{hor_domega}
d\omega_s(X,Y,Z)=0.
\end{equation}
\end{prop}

\begin{proof}
If \eqref{str_eq_mod} are satisfied, then we have
$
0\ =\ d(d\eta_i)\ =\ d\Big(2\omega_i+
S\eta_j\wedge\eta_k\Big),
$
which  implies \eqref{hor_domega}.

Conversely, suppose the given qc-structure is locally defined by  1-form $(\eta_1,\eta_2,\eta_3)$ which satisfies \eqref{hor_domega}. By \eqref{streq} we have $
\Big(\omega_j\wedge\alpha_k-\omega_k\wedge\alpha_j\Big)|_{H}=0,
$
which after a contraction with the endomorphism $I_i$ gives
\begin{multline*}
0=(\omega_j\wedge\alpha_k-\omega_k\wedge\alpha_j)(X,e_a,I_ie_a)\ =\omega_j(X,e_a)\alpha_k(I_ie_a)+\omega_j(e_a,I_ie_a)\alpha_k(X)+
\omega_j(I_ie_a,X)\alpha_k(e_a) \\ -\omega_k(X,e_a)\alpha_j(I_ie_a)-\omega_k(e_a,I_ie_a)\alpha_j(X)-
\omega_k(I_ie_a,X)\alpha_j(e_a)
=2\omega_j(X,e_a)\alpha_k(I_ie_a) - 2\omega_k(X,e_a)\alpha_j(I_ie_a)\\ =
2\alpha_k(I_kX)\ +\ 2\alpha_j(I_jX).
\end{multline*}
Since the above calculation is valid for any even permutation $(i,j,k)$, it follows that $\alpha_s(X)=0$ which completes the  proof of the Proposition.
\end{proof}

\section{The related Riemannian geometry}\label{five}

{A (4n + 3)-dimensional
(pseudo) Riemannian manifold $(M,g)$  is 3-Sasakian  if the cone metric  is a (pseudo) hyper-K\"ahler metric \cite{BG,BGN}.  We note explicitly that in this paper 3-Sasakian manifolds  are to be understood in the wider sense of positive (the usual terminology) or negative 3-Sasakian structures, cf. \cite[Section 2]{IV2} and \cite[Section 4.4.1]{IV3} where the "negative" 3-Sasakian term was adopted in the case when the Riemannian cone is hyper-K\"ahler of signature $ (4n,4)$. Every 3-Sasakian manifold is a qc-Einstein manifold of constant qc-scalar curvature, \cite{Biq1},  \cite{IMV} and \cite{IV2}.  As well known, a positive 3-Sasakian manifold is Einstein with a positive Riemannian scalar curvature \cite{Kas} and, if complete,  it is compact with finite fundamental group due to Myer’s
theorem.  The negative 3-Sasakian structures are Einstein with respect to the corresponding pseudo-Riemannian metric of signature $(4n,3)$ \cite{Kas,Tan}.  In this case, by a simple change of signature, we obtain a positive definite $nS$ metric on $M$,  \cite{Tan,Jel,Kon}.}

 By \cite[Theorem 1.3]{IMV} when $Scal>0$, and \cite{IV2} and \cite[Theorem 4.4.4]{IV3} when $Scal<0$ a qc-Einstein of dimension at least eleven is locally qc-homothetic to a 3-Sasakian structure. The corresponding result  in the seven dimensional case was proven with the extra assumption that the qc-scalar is constant. Thanks to Theorem \ref{t:main} the additional hypothesis is redundant, hence we have the following
\begin{cor}\label{c:3-sasakian}
 A seven dimensional qc-Einstein manifold  of nowhere vanishing qc-scalar curvature is  locally qc-homothetic to a 3-Sasakian structure.
\end{cor}
There are many known examples of positive 3-Sasakian manifold, see \cite{BG} and references therein for a nice overview of 3-Sasakian spaces. On the other hand, certian SO(3)-bundles over quaternionic K\"ahler manifolds with negative scalar curvature constructed in \cite{Kon,Tan,Jel} are examples of negative 3-Sasakian manifolds.
Other, explicit examples of negative 3-Sasakian manifolds are constructed also in \cite{AFIV}.

Complete and regular 3-Sasakian manifolds, resp.  $nS$-structures, fiber over a quaternionic K\"ahler manifold with positive, resp. negative, scalar curvature \cite{Is, BGN,Tan,Jel} with fiber $SO(3)$. Conversely, a quaternionic K\"ahler manifold with positive (resp. negative) scalar curvature has a canonical $SO(3)$ principal bundle, the total space of which admits a natural 3-Sasakian (resp. $nS$-) structure \cite{Is,Kon,Tan, BGN, Jel}.

In this section  we describe the properties of qc-Einstein structures of zero qc-scalar curvature, which complement the well known results in the 3-Sasakian  case.  A common feature of the  $Scal=0$ and $Scal\ne 0$ cases is the existence of Killing vector fields.
\begin{lemma}\label{l:killing}
Let $M$ be a qc-Einstein manifold with zero qc-scalar curvature. If $(\eta_1,\eta_2,\eta_3)$ is an $\mathbb{R}^3$-valued local 1-form defining the qc structure as in~\eqref{str_eq_mod}, then the corresponding Reeb vector fields $\xi_1,\xi_2,\xi_3$ are Killing vector fields for the Riemannian metric $h$, cf. \eqref{e:Riem-metric}.
\end{lemma}

\begin{proof}
By Theorem \ref{str_eq_mod_th} c) we have $\alpha_i=0$, hence
$\nabla_{\A}\xi_i=0$ while  Lemma~\ref{Einstein:RT} yields $T(\xi_s,\xi_t)=0$. Therefore,
$
[\xi_s,\xi_t]=\nabla_{\xi_s}\xi_t-\nabla_{\xi_t}\xi_s-T(\xi_s,\xi_t)=0,
$
which implies for any $i,s,t \in\{1,2,3\}$  we have
$
 ({\mathcal
L_{\xi_i}}h)(\xi_s,\xi_t)=-h([\xi_i,\xi_s],\xi_t)-h(\xi_s,[\xi_i,\xi_t])
 =  0.
$
Furthermore, using $d\eta_j(\xi_i,X)=\alpha_k(X)=0$ we compute
\begin{equation*}
({\mathcal
L}_{\xi_s}h)(\xi_t,X)=-h(\xi_t,[\xi_s,X])=d\eta_t(\xi_s,X)=0.
\end{equation*}
Finally, \eqref{to} gives $({\mathcal L}_{\xi_i}h)(X,Y)=({\mathcal
L}_{\xi_i}g)(X,Y)=2T^0_{\xi_i}(X,Y)=0$, which completes the proof.
\end{proof}

\subsection{The quotient space of a qc Einstein manifold with $S=0$}

The total space of an $\mathbb{R}^3$-bundle  over a hyper-K\"ahler manifold with closed and locally exact K\"ahler forms $2\omega_s = d\eta_s$ with connection
1-forms $\eta_s$ is a qc-structure  determined by the three 1-forms $\eta_s$, which is qc-Einstein of vanishing qc-scalar curvature, see \cite{IV2}. In fact, we  characterize qc-Einstein manifold with vanishing qc-scalar curvature as $\mathbb{R}^3$-bundle over hyper-K\"ahler manifold.

Let  $M$ be a qc-Einstein manifold. As observed in Corollary  \ref{c:vert int} and the paragraph after it  the vertical distribution $V$ is completely integrable hence defines a foliation on $M$.  We recall, taking into account \cite{Pal},  that the quotient space $P=M/V$ is a manifold when the foliation is regular and the quotient topology is Hausdorff.

{If $P$ is a manifold and all the leaves of $V$ are compact, then by Ehresmann's fibration theorem  \cite{Ehr, Pal} it follows that $\Pi:M\rightarrow P$ is a locally trivial fibration and  all the leaves are  isomorphic.  By \cite{Pal},  examples of such foliations are given by  regular foliations   on compact manifolds.  In the case of  a qc-Einstein manifold of non-vanishing qc-scalar curvature,  the leaves of the foliation generated by $V$ are  Riemannian 3-manifold of positive constant curvature. Hence, if the associated  (pseudo) Riemannian metrics on $M$ is complete, then the leaves of the foliation are compact. On the other hand, in the case of vanishing qc-scalar curvature, the leaves of the foliation are flat Riemannian manifolds that may not be compact  as  is, for example,  the case of the quaternionic Heisenberg group.} We summarize the properties of the Reeb foliation on a qc-Einstein manifold of vanishing qc-scalar curvature case in the following

\begin{prop}\label{t:hKquot} Let $M$ be a qc-Einstein manifold with zero qc-scalar curvature.
\begin{enumerate}[a)]
\item If the vertical distribution $V$ is regular and the space of leaves $P=M/V$ with the quotient topology is  Hausdorff, then  $P$ is a locally hyper-K\"ahler manifold.
\item If the leaves of the foliation generated by $V$ are compact then there exists an open dense subset $M_o\subset M$ such that  $P_o:=M_o/V$ is a locally hyper-K\"ahler manifold.
    \end{enumerate}
\end{prop}

\begin{proof} We begin with the proof of a).
 By Theorem~\ref{str_eq_mod_th} we can assume, locally, the structure equations  given in Theorem
\ref{str_eq_mod_th}. This, together with \cite[Lemma 3.2 \&
Theorem 3.12]{IMV} imply  that the horizontal metric $g$, see also
\eqref{to}, and the {closed} local fundamental 2-forms
$\omega_s$, see \eqref{str_eq_mod} with $S=0$,  are projectable.  The claim of part a)  follows from
Hitchin's lemma \cite{Hit}.

We turn to the proof of part b).
Lemma \ref{l:killing} implies that, in particular, the Riemannian metric $h$ on $M$  is bundle-like, i.e., for any two horizontal vector fields $X$ and $Y$ in the normalizer of $\mathcal V$ under the Lie bracket, the equation $\xi h(X,Y)=0$ holds for any vector field $\xi$ in $\mathcal V$. Since all the leaves of the vertical foliation are assumed to be compact, we can apply \cite[Proposition 3.7 ]{Mo},
which shows that
$P=M/V$ is a 4n-dimensional orbifold. In particular $P$ is a Hausdorff space. The regular points of any orbifold are an open dens set. Thus, if we let $P_o$ to be the set of all regular points of $P$, then $P_o$ is an open dens subset of $P$ which is also a manifold. It follows that if $M_o:=\Pi^{-1}(P_o)$ then all the leaves of the restriction of the vertical foliation to $M_o$ are regular and hence the claim of b) follows.
\end{proof}

\subsection{The Riemannian curvature}

Let  $M$ be a qc-Einstein manifold.  Note that, by applying an appropriate qc homothetic transformation, we can aways reduce a  general qc-Einstein structure to one whose normalized qc-scalar curvature $S$ equals $0,2$ or -2. Consider the one-parameter family of (pseudo)  Riemannian metrics  $g^{\lambda},\
\lambda\ne 0$ on  $M$ by letting
$h^{\lambda}(\A,\B)\ :=\ h(\A,\B)+(\lambda-1)h|_V.$ Let $\nabla^\lambda$ be the Levi-Civita connection of $h^\lambda.$
Note that $h^{\lambda}$ is a positive-definite metric when $\lambda>0$ and
has signature $(4n,3)$ when $\lambda<0$.

Let us recall that, if $S=2$ and  $\lambda=1$  the Riemannian metric $h=h^\lambda$ is a 3-Sasakian metric on $M$. In particular, it is an Einstein metric of positive Riemannian scalar curvature (4n + 2)(4n + 3) \cite{Kas}. There is  also a second Einstein metric, the "squashed" metric, in the family $h^\lambda$ when $\lambda={1}/{(2n+3)}$, see \cite{BG}. The case $S=-2$ is completely analogous. Here we have two distinct pseudo-Riemannian Einstein metrics  corresponding to  $\lambda=-1$ and $\lambda=-{1}/{(2n+3)}$. The first one defines a negative 3-Sasaskian structure.  On the other hand, the metric $h^{\lambda}$ with $ \lambda =1$ (assuming $S=-2$) gives an $nS$ structure on $M$. In \cite{Tan}, it was shown that the  Riemannian
Ricci tensor of the latter has precisely two constant eigenvalues, $ - 4n - 14$  (of multiplicity $4n$) and
$4n + 2$  (of multiplicity $3$),  and that the Riemannian scalar curvature is the negative constant $-16n^2 - 44n + 6$. In particular, in this case,  $(M,h^{\lambda})$ is an example of an A-manifold in the terminology of \cite{Gr}.

The following proposition  addresses the case $S=0$. However, the argument is valid for all values of $S$ and $\lambda\ne 0$. In particular, we obtain new proofs of the above mentioned results concerning the cases of positive and negative 3-Sasakian structures.

\begin{prop}\label{p:einst m} Let $M$ be a qc-Einstein manifold with normalized qc-scalar curvature  ${S}$. For a vector field $A$, let $[A]_V$ denote the orthogonal projection of $A$ to the vertical space $V$.

%\item
The (pseudo) Riemannian Ricci and scalar curvatures of  $h^{\lambda}$ are given by
\begin{align} \label{Ric-lambda}
Ric^{\lambda}(\A,\B)
&=\Big(4n\lambda+\frac{S^2}{2\lambda}\Big)h^\lambda\Big([\A]_V,[\B]_V\Big)+\Big(2{S}(n+2)-6\lambda\Big)h^\lambda\Big([\A]_H,[\B]_H\Big)\\
Scal^{\lambda} &= \frac{1}{\lambda}\Big(-12n\lambda^2+8n(n+2)S\lambda+\frac{3}{2}{S}^2\Big).
\end{align}
%\end{enumerate}
In particular, if $S=0$, the Ricci curvature of each metric in the family $h^{\lambda}$ has exactly two different constant eigenvalues of multiplicities $4n$ and $3$ respectively.
\end{prop}

\begin{proof}
We start by noting that the difference $L=\nabla^{\lambda}-\nabla$  between the Levi-Cevita connection $\nabla^{\lambda}$ and the Biquard connection $\nabla$ is given by
\begin{equation}\label{nabla-lambda}
L(A,B)\equiv \nabla^{\lambda}_{\A}\B-\nabla_{\A}\B= \frac{S}{2}[\A]_V\times [\B]_V + \sum_{s=1}^3\Big\{-\omega_s(\A,\B)\xi_s+\lambda\eta_s(\A)I_s\B+\lambda\eta_s(\B)I_s\A \Big\}.
\end{equation}
Indeed, if   we let $D_\A\B:=\nabla_\A\B+L(\A,\B)$, then $h^\lambda(L(\A,\B), \C)$ is  skew symmetric in $\B$ and $\C$, hence the connection $D$ preserves the metric $h^\lambda$. Furthermore, the torsion tensor of $D$
vanishes since
$h^\lambda(L(\A,\B),\C)-h^\lambda(L(\B,\A),\C)=-h^{\lambda}(T(\A,\B),\C).$ The latter follows from the formula for $T$ in Lemma~\ref{Einstein:RT}.   Thus $D$  is the Levi-Civita connection of $h^\lambda$.

The well known formula for the difference $R^{\lambda}-R$ between the curvature tensors of two connections $\nabla^{\lambda}$ and $\nabla$ gives
\begin{equation}\label{R-R2}
R^{\lambda}({\A},{\B})\C\ -\ R({\A},{\B})\C\ =\ (\nabla_{\A}L)({\B},\C)\ -\ (\nabla_{\B}L)({\A},\C)\ +\  [L_{\A},L_{\B}]\C+L(T({\A},{\B}),\C).
\end{equation}
From \eqref{nabla-lambda}, it follows $L$  is  $\nabla$-parallel. Thus,
in the right hand side of the above formula only the last two terms are non-zero.  Furthermore, we have that
$[L_{\A},L_{\B}]\C=L(\A, L(\B,\C))-L(\B, L(\A,\C))$. A straightforward
computation  gives
\begin{multline}\label{R_lambda}
R^\lambda(A,B)C = R(A,B)C + h^{\lambda}\Big([B]_V,[C]_V\Big)\Big(\frac{S^2}{4\lambda}[A]_V+\lambda[A]_H\Big) -  h^{\lambda}\Big([A]_V,[C]_V\Big)\Big(\frac{S^2}{4\lambda}[B]_V+\lambda[B]_H\Big)\\
+\sum_{(i,j,k)-\text{cyclic}}\Big\{\ \Big(\frac{S}{2}-\lambda\Big)\eta_k(A)\omega_j(B,C)\ -\ \Big(\frac{S}{2}-\lambda\Big)\eta_k(B)\omega_j(A,C)
\\
 -\ \Big(\frac{S}{2}-\lambda\Big)\eta_j(A)\omega_k(B,C)\ +\ \Big(\frac{S}{2}-\lambda\Big)\eta_j(B)\omega_k(A,C)\  + \ (S+2\lambda)\eta_k(C)\omega_j(A,B)
\\
 -\ (S+2\lambda)\eta_j(C)\omega_k(A,B)\ -\ \lambda\eta_i(B)h^\lambda\Big([A]_H,[C]_H\Big)\ +\ \lambda\eta_i(A) h^\lambda\Big([B]_H,[C]_H\Big)\ \Big\}\xi_i\\
\ +\sum_{(i,j,k)-\text{cyclic}} \Big\{\ \Big(\frac{\lambda S}{2}-\lambda^2 \Big)\eta_j\wedge\eta_k(B,C) I_iA
-\ \Big(\frac{\lambda S}{2}-\lambda^2 \Big)\eta_j\wedge\eta_k(A,C) I_iB
\\-\ (\lambda S-2\lambda^2)\eta_j\wedge\eta_k(A,B) I_iC
-\
\lambda\omega_i(B,C)I_iA
\ +\ \lambda\omega_i(A,C)I_iB\ +\ 2\lambda\omega_i(A,B)I_iC \ \Big\}.
\end{multline}
After taking the  trace with respect to $A$ and $D$ in equation \eqref{R_lambda}, we obtain
$$Ric^{\lambda}(\B,\C)=Ric\Big([\B]_H,[\C]_H\Big)+\Big(4n\lambda+\frac{S^2}{2\lambda}\Big)h^\lambda\Big([\B]_V,[\C]_V\Big)-6\lambda h^\lambda\Big([\B]_H,[\C]_H\Big).$$
Since $M$ is assumed to be qc Einstein, we have
\begin{align}
Ric\Big([\B]_H,[\C]_H\Big)=\frac{Scal}{4n}g\Big([\B]_H,[\C]_H\Big)=2(n+2){S} h^\lambda\Big([\B]_H,[\C]_H\Big),
\end{align} which yields~\eqref{Ric-lambda}. Taking one more  trace in \eqref{Ric-lambda} gives the formula for the scalar curvature.
\end{proof}

\end{document}